\documentclass[11pt]{amsart}
\usepackage{epsfig,latexsym,amsfonts,amsmath,amssymb,color}
\usepackage{graphicx}
\def\dvg{\rm div}

\def\IntO{\int\limits_{\Omega}}

\def\Rd{\mathbb R^d}

\newtheorem{Proposition}{Proposition}[section]
\newtheorem{Theorem}{Theorem}[section]
\newtheorem{Remark}{Remark}[section]

 \numberwithin{equation}{section}
\newcommand\be{\begin{eqnarray*}}
\newcommand\ee{\end{eqnarray*}}
\newcommand\ben{\begin{eqnarray}}
\newcommand\een{\end{eqnarray}}

 \def\Normt#1{|\!|\!| #1 |\!|\!|}   

\def\Nor2#1{\langle\!\langle #1 \rangle\!\rangle}
\def\mean#1{\{\!\!\{\,#1\,\}\!\!\}}

\def\wh{\widehat}
\def\wt{\widetilde}

\def\cP{\mathcal P}
\def\cO{{\mathcal O}}
\def\cOD{{\mathcal O}^D}
\def\cON{{\mathcal O}^N}
\def\cOI{{\mathcal O}^I}

\def\dvg{{\rm div}}

\def\IntO{\int\limits_\Omega}

\begin{document}

\title{Exact constants in Poincar\'e type inequalities
 for functions with zero mean  boundary traces}

\author{A. I. Nazarov}\thanks
{V.A. Steklov Institute of Mathematics in St. Petersburg,
191023, Fontanka 27, St.-Petersburg, Russia
and St. Petersburg State University,
198504, Universitetskii pr. 28, St.-Petersburg, Russia.}
\author{S. I. Repin}
\thanks{V.A. Steklov Institute of Mathematics in St. Petersburg,
191023, Fontanka 27, St.-Petersburg, Russia
and University of Jyv\"askyl\"a, P.O. Box 35 (Agora), FIN-40014 , Jyv\"askyl\"a,  Finland}

\begin{abstract}
In this paper, we investigate
Poincar\'e type inequalities for the functions having zero mean value
on the whole boundary of a Lipschitz domain  or on a measurable part of the boundary.
We find exact and easily computable constants in these inequalities
 for some basic domains (rectangles,
cubes, and right triangles) and discuss  applications of the inequalities
to quantitative analysis of partial differential equations.
\end{abstract}

\maketitle

\section{Introduction}
Let $\Omega$ be a connected bounded domain in $\Rd$ with Lipschitz boundary $\partial\Omega$.
The classical Poincar\'e inequality 
reads
\begin{eqnarray}
&&\label{1.1}
\|w\|_{2,\Omega}\leq C_P(\Omega)\|\nabla w\|_{2,\Omega},\qquad
\forall w\in \widetilde H^1(\Omega),
\end{eqnarray}
where
$$
\widetilde H^1(\Omega):=\left\{w\in H^1(\Omega)\,\mid\,\mean{w}_\Omega=0\right\}.
$$
Here and later on $\mean{g}_\omega$ denotes the mean value of $g$ on the set $\omega$ 
while $\|\cdot\|_{2,\omega}$ stands for the norm in $L^2(\omega)$.

The inequality (\ref{1.1}) was established by H. Poincar\'e (\cite{Poin1}, \cite{Poin2})
for a certain class of domains. V.A. Steklov \cite{St1} found that the sharp constant
in (\ref{1.1}) is equal to
$\lambda^{-\frac 12}$, where $\lambda$ is the smallest positive eigenvalue of the problem
\begin{equation}\label{1.3}
\aligned
-\Delta u=&\lambda u&&\mbox{in}\quad \Omega;\\
\partial_{\bf n}u=&0&&\mbox{on}\quad \partial\Omega.
\endaligned
\end{equation}
L.E.~Payne and H.F.~Weinberger \cite{PW} proved that for convex domains in $\Rd$ this constant
satisfies the estimate
$
C_P(\Omega)\leq \frac{{\rm diam}\Omega}{\pi}.
$
Estimates of constants in Poincar\'e type inequalities have been studied
by many authors (see, e.g.,  \cite{ENT,Fuchs,Ka,KiLi, NKP} and other publications cited therein).

In this paper,
 we consider estimates similar to  (\ref{1.1}), for the functions having zero mean
 on a certain part of the boundary (or on the whole boundary), namely,
\begin{eqnarray}
&&\label{eq1Omega}
\|w\|_{2,\Omega}\leq C_1(\Omega,\Gamma)\|\nabla w\|_{2,\Omega},\qquad
\forall w\in \wt H^1(\Omega,\Gamma),\\
\label{eq1Gamma}
&&\|w\|_{2,\Gamma}\leq C_2(\Omega,\Gamma)\|\nabla w\|_{2,\Omega},\qquad
\forall w\in\wt  H^1(\Omega,\Gamma),
\end{eqnarray}
where $\Gamma$ is a measurable part of $\partial\Omega$ (we assume that $(d-1)$-measure of
$\Gamma$ is positive),
\begin{eqnarray*}
\wt H^1(\Omega,\Gamma)=\left\{w\in H^1(\Omega)\,\mid\,\mean{w}_{\Gamma} =0\right\}.
\end{eqnarray*}
Since the quantity $\|\nabla w\|_{2,\Omega}+|\int_\Gamma w\,ds|$ is a norm equivalent to the
original norm of $H^1(\Omega)$, existence of the constants $C_1(\Omega,\Gamma)$ and
$C_2(\Omega,\Gamma)$ is easy to prove.

We recall  that the extremal
function in (\ref{eq1Omega}) is an eigenfunction
$u\in \wt H^1(\Omega,\Gamma)$ of the boundary value problem
\begin{equation}
\aligned
\label{1.5}
-\Delta u=\,\lambda u\qquad&\mbox{in}\quad \Omega;\\
\partial_{\bf n}u=\,\mu\equiv-\frac {\lambda}{|\Gamma|}\int_{\Omega} u\,dx\quad&
{\rm on} \;\quad\Gamma;\\
\partial_{\bf n}u=0\quad&\mbox{on}\quad  \partial\Omega\setminus\Gamma
\endaligned
\end{equation}
corresponding to the least eigenvalue $\lambda>0$.
This fact  follows from standard variational arguments.

Analogously, the extremal function in (\ref{eq1Gamma}) is an eigenfunction
$u\in \wt H^1(\Omega,\Gamma)$ of the boundary value problem
\begin{equation}\label{1.6}
\aligned
\Delta u=\, &0\qquad\mbox{in}\quad\Omega;\\
\partial_{\bf n}u=\, &\lambda u \quad\mbox{on}\quad  \Gamma;
\qquad \partial_{\bf n}u=0\quad\mbox{on}\quad  \partial\Omega\setminus\Gamma,
\endaligned
\end{equation}
which corresponds to the least positive eigenvalue.\medskip

The problem (\ref{1.6}) in the case $\Gamma=\partial\Omega$ was introduced by V. A. Steklov
in \cite{St}. We note also that for $n = 2$ $(n = 3)$ and a particular choice of $\Gamma$,
eigenvalues of the problem (\ref{1.6}) give so-called {\it sloshing frequencies} of free
oscillations of a liquid in a channel (container, respectively); see, for example, \cite{Lamb},
ch. IX.\medskip

It is easy to show that the eigenfunctions
of the problems (\ref{1.5}) and (\ref{1.6}) form complete orthogonal systems in
$L^2(\Omega)$ and in $L^2(\Gamma)$, respectively, and the exact constants
in (\ref{eq1Omega}) and (\ref{eq1Gamma}) are equal to $\lambda^{-\frac 12}$, where
$\lambda$ is the minimal positive eigenvalue. For simply connected domains in ${\mathbb R}^2$,
 eigenvalues of the problem (\ref{1.6})  were estimated from above in  \cite{We,HPS}
 (in fact these estimates are valid for a wider class of problems, which contain
 a positive weight function on $\Gamma$).
 In \cite{DS}, it was shown that for doubly connected domains in ${\mathbb R}^2$ the
 symmetrization method yields certain estimates of the smallest positive
 eigenvalue from below. \medskip

Our goal is to find exact values of $ C_1(\Omega,\Gamma)$ and $ C_2(\Omega,\Gamma)$
for certain basic domains (rectangular domains and  right  triangles).
Finding the constants amounts to finding the corresponding minimal positive eigenvalues.
In some cases, this goal can be achieved by standard methods.
For example, if $\Omega$ is a rectangle (parallelepiped) and $\Gamma$ coincides
with one side of $\Omega$, then all the eigenfunctions of (\ref{1.5}) and (\ref{1.6})
can be constructed explicitly (by means of the separation of variables method).

However, these type arguments do not work in other cases because
it is impossible to construct explicitly {\em all} eigenfunctions. The goal can be achieved
 if we find a suitable eigenfunction and  prove that it
indeed corresponds to the minimal positive eigenvalue. Proving this fact
requires rather sophisticated argumentation, which combines analytical
and geometrical analysis. Moreover, the proof strongly depends
on the domain and boundary conditions so that  arguments used for rectangular
domains and $\Gamma=\partial\Omega$ considered in Section 2
 differ  from those used for triangles considered in Section 3.

Section 4  presents  an example, which shows that the constants are valuable
for quantitative analysis of differential equations. We consider
two elliptic boundary value problems with different boundary conditions and source terms. The
second problem is a certain simplification of the first one, namely,  if the functions presenting
source terms and Dirichlet or Neumann boundary conditions are complicated, then they are replaced
by simple (piecewise constant or  affine) functions. Such a simplified problem is more convenient
for numerical methods because simplified boundary conditions can be exactly satisfied by simple
approximations and elementwise integrals can be sharply computed by simple quadrature formulas.
We deduce an easily computable bound of the difference between two exact solutions.
It contains constants in (\ref{eq1Omega}), (\ref{eq1Gamma}), and Poincar\'e
 inequalities. Computation of the bound is simple and
 needs only integration of known functions. In practice, this estimate
is useful if a certain desired tolerance level is stated a priori.
The estimate
suggests possible simplifications of data and  suitable
meshes
for which the difference is smaller than the tolerance level. Then, in computations
(e.g., in the finite element method) we can use the simplified
problem instead of the original one.
 Our analysis is performed for a particular linear
elliptic equation, but similar arguments  lead to similar estimates for other
differential equations associated with the pair of conjugate operators ${\rm grad}$ and $-\dvg$.
Other applications of  (\ref{eq1Omega}) and (\ref{eq1Gamma}) are related to a posteriori error
estimation methods for partial differential equations, where computable bounds between exact
solutions and approximations often involve constants in Poincar\'e type inequalities (see
\cite{ReGruyter}).

\section{Exact constants for rectangles and parallelepipedes}
In this section, we deal with rectangles and parallelepipedes. First, we consider the simplest
 case where $\Gamma$ coincides with one side of $\Omega$.

\begin{Proposition}
\label{rectside}
{\bf 1}. Let $\Omega=(0,h_1)\times (0,h_2)$ and $\Gamma=\{x_1=0,x_2\in [0,h_2]\}$. Then the
sharp constants in  (\ref{eq1Omega}) and (\ref{eq1Gamma})  are equal to
$\frac 1{\pi}\max\{2h_1;h_2\}$ and
$\left(\frac{\pi}{h_2}\tanh(\frac{\pi h_1}{h_2})\right)^{-\frac 12}$,
respectively.\medskip

{\bf 2}. Let $\Omega=(0,h_1)\times (0,h_2)\times (0,h_3)$ and
 $\Gamma=\partial\Omega\cap\{x_1=0
\}$.
 Then the sharp constants in
(\ref{eq1Omega}) and (\ref{eq1Gamma}) are equal to
$$\frac 1{\pi}\max\{2h_1;h_2;h_3\}\quad{\rm and}\quad\left(\frac{\pi}{\max\{h_2;h_3\}}
\tanh(\frac{\pi h_1}{\max\{h_2;h_3\}})\right)^{-\frac 12},
$$ respectively.
\end{Proposition}

\begin{proof}
In this case, all the eigenfunctions of the problems
(\ref{1.5}) and (\ref{1.6}) can be found by the separation of variables. For instance,
the eigenfunctions of (\ref{1.6}) in the rectangle are defined by the relation
\begin{equation*}
u_k(x)=\cos\left(\frac{\pi k}{h_2}x_2\right)\cosh\left(\frac{\pi k}{h_2}(x_1-h_1)\right),\qquad
k=0,1,2,\dots
\end{equation*}
and evidently form a complete orthogonal system in $L^2(\Gamma)$. Therefore, the corresponding
least eigenvalue of the problem (\ref{1.6}) is
$$\lambda_1=\frac{\pi}{h_2}\tanh(\frac{\pi h_1}{h_2}).
$$

The proof of other statements is quite similar, and we omit it. \end{proof}

Now we turn to the case $\Gamma=\partial\Omega$. The problem in a rectangle is symmetric
with respect to two axes. Therefore,  it is convenient to select the coordinate system such that
$\Omega=(-\frac{h_1}2,\frac{h_1}2)\times (-\frac{h_2}2,\frac{h_2}2)$  (see Fig.~1).

\begin{Theorem}
\label{rectfull1}
Let $\Omega=(-\frac{h_1}2,\frac{h_1}2)\times (-\frac{h_2}2,\frac{h_2}2)$ and
$\Gamma=\partial\Omega$. Then the sharp constant in (\ref{eq1Omega}) is equal
to $\frac 1{\pi}\max\{h_1;h_2\}$.
\end{Theorem}
\begin{proof}
Due to the biaxial symmetry all the eigenfunctions of (\ref{1.5})
and (\ref{1.6}) can be chosen either even or odd with respect to the axes
$x_1$ and $x_2$.
First, we consider the eigenfunctions of (\ref{1.5}), which are odd with respect to $x_1$.
In this case, $\mu=0$ and we arrive at the following problem:
\begin{equation}\label{eq1.1odd}
\aligned
-\Delta u=&\lambda u&&\mbox{in}\quad\Omega^+:=(0,\frac{h_1}2)\times (-\frac{h_2}2,\frac{h_2}2),\\
 u=&0 &&\mbox{on}\quad \{x_1=0\}\cap\Omega,\qquad
\partial_{\bf n}u=0\quad\mbox{on}\quad  \partial\Omega^+\setminus\{x_1=0\}.
\endaligned
\end{equation}
It is easy to see that the eigenfunctions  of (\ref{eq1.1odd})
are defined by the relations
\begin{equation*}
\aligned
u_{km}^{(1)}(x)&=\sin\left(\frac{\pi(2k+1)}{h_1}x_1\right)
\cos\left(\frac{2\pi m}{h_2}x_2\right),& \\
u_{km}^{(2)}(x)&=\sin\left(\frac{\pi(2k+1)}{h_1}x_1\right)
\sin\left(\frac{\pi(2m+1)}{h_2}x_2\right),& k,m=0,1,\dots\;.
\endaligned
\end{equation*}
 They form a system
of orthogonal functions, which is complete  in $L^2(\Omega^+)$. Hence, we conclude that the least
eigenvalue of the problem (\ref{eq1.1odd}) is
$\lambda_{00}^{(1)}=\left(\frac{\pi}{h_1}\right)^2.$

\begin{figure}[h!]
\label{domain+}
\centerline{\includegraphics[width=2.1in,height=1.24in]{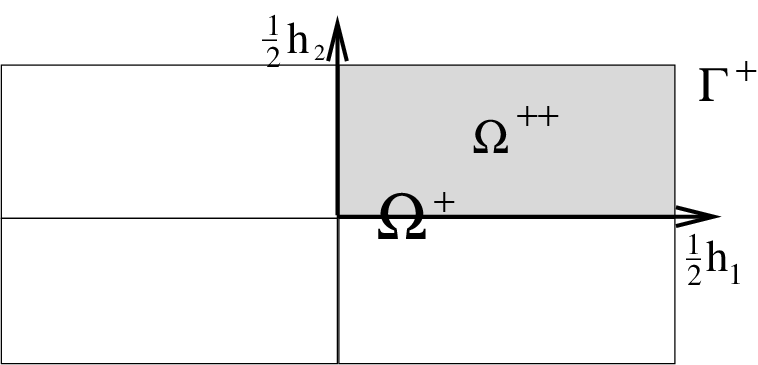}}
\caption{ $\Omega$, $\Omega^+$,  $\Omega^{++}$, and $\Gamma^+$.}
\end{figure}

All eigenfunctions of the problem (\ref{1.5}), which are odd with respect to $x_2$
can be constructed quite similarly and the corresponding least eigenvalue
is
$\left(\frac{\pi}{h_2}\right)^2$.

It remains to consider the eigenfunctions, which are even with respect to both variables.
 They belong to the space $\wt H^1(\Omega^{++},\Gamma^+)$, where
\begin{equation*}
\Omega^{++}:=\left(0,\frac{h_1}{2}\right)\times \left(0,\frac{h_2}{2}\right)\qquad
{\rm and}\qquad\Gamma^+=\Gamma\cap\partial\Omega^{++}.
\end{equation*}
In this case, we need to
solve the problem
\begin{equation}\label{eq1.1even}
\aligned
-\Delta u&=\lambda u &&\mbox{in}\quad  \Omega^{++};\\
\partial_{\bf n}u&=\mu &&\mbox{on}\quad  \Gamma^+;
\qquad \partial_{\bf n}u=0\quad  \mbox{on}\quad  \partial\Omega^{++}\setminus\Gamma^+.
\endaligned
\end{equation}
The eigenvalues $\lambda_k^e$ of the problem (\ref{eq1.1even})
(enumerated in the increasing order
and repeated according to their multiplicity)
 are critical values of the Rayleigh quotient
\begin{eqnarray}
\label{quotient++}
Q[v]\equiv\frac {\|\nabla v\|_{2,\Omega^{++}}^2}{\|v\|_{2,\Omega^{++}}^2},
\end{eqnarray}
on the space $\wt H^1(\Omega^{++},\Gamma^+)$.

Consider the functional $Q$ on the whole space $H^1(\Omega^{++})$.
By the variational principle (see, e.g., \cite{BS3}, (1.15)),
its critical values $\widetilde\lambda_k^e$
enumerated in the increasing order and repeated in accordance with
  their multiplicity satisfy
the relation\footnote{Note that $\wt H^1(\Omega^{++},\Gamma^+)$
has codimension $1$ in $H^1(\Omega^{++})$.}
$\widetilde\lambda_k^e\le \lambda_k^e\le \widetilde\lambda_{k+1}^e$.
Therefore, if there exists an eigenvalue of the
problem (\ref{eq1.1even}) in the interval
 $(\widetilde\lambda_0^e,\widetilde\lambda_1^e)$,
 then it must be $\lambda_0^e$.

Note that $\widetilde\lambda_k^e$ are eigenvalues
of the conventional Neumann problem
\begin{equation*}
-\Delta u=\lambda u\quad  \mbox{in}\quad  \Omega^{++};\qquad \partial_{\bf n}u=0\quad
\mbox{on}\quad  \partial\Omega^{++},
\end{equation*}
and thus, $\widetilde\lambda_0^e=0$, $\widetilde\lambda_1^e=\min\{\big(\frac{2\pi}{h_1}\big)^2;
\big(\frac{2\pi}{h_2}\big)^2\}$.

Let us consider  the equation
\begin{equation}
\label{cot}
\frac {h_1}2\cot\Big(\frac {\omega h_2}2\Big)+
\frac {h_2}2\cot\Big(\frac {\omega h_1}2\Big)+\frac 2{\omega}=0.
\end{equation}
It is easy to see that  the function in the
left-hand side of (\ref{cot}) decreases from $+\infty$ to $-\infty$ in the interval
 $(0,\min\{\frac{2\pi}{h_1};\frac{2\pi}{h_2}\})$,
so that the equation
has a unique solution $\omega_0$ in this interval. The equation (\ref{cot})
arises in our analysis by the following reasons.
Direct calculation shows that the function
$$
v_0(x)=\frac {\cos(\omega_0 x_1)}{\sin(\frac {\omega_0 h_1}2)}+
\frac {\cos(\omega_0 x_2)}{\sin(\frac {\omega_0 h_2}2)}
$$
solves the problem (\ref{eq1.1even}) with $\lambda=\omega_0^2$.
Then, the condition $\int_{\Gamma_+}v_0\,ds=0$ infers the equation  (\ref{cot}).

Thus, we conclude that $\lambda_0^e=\omega_0^2$. However, it is easy to see that
$$\omega_0>\min\{\frac{\pi}{h_1};\frac{\pi}{h_2}\}.
$$
Therefore, the least eigenvalue of the problem (\ref{1.5}) is
$\min\{\big(\frac{\pi}{h_1}\big)^2;\big(\frac{\pi}{h_2}\big)^2\}$,
and the statement follows.\end{proof}

\begin{Theorem}
\label{rectfull2}
 If $\Omega=(-\frac{h_1}2,\frac{h_1}2)\times (-\frac{h_2}2,\frac{h_2}2)$ and
 $\Gamma=\partial\Omega$, then the smallest constant $C_2(\Omega,\Gamma)$
in (\ref{eq1Gamma}) equals
$\left(\frac {2z_0(\alpha_0)}{\sqrt{h_1h_2}}
\tanh(\frac {z_0(\alpha_0)}{\alpha_0})\right)^{-\frac 12}$, where
$\alpha_0=\sqrt{\frac {\max\{h_1;h_2\}}{\min\{h_1;h_2\}}}$ and
$z_0(\alpha)$ is a unique root
of the equation
\begin{equation}
\label{tanh}
\tanh\left(\frac {z}{\alpha}\right)\tan(z\alpha)=1,
\end{equation}
such that $0<z_0\alpha<\frac {\pi}2$.
\end{Theorem}
\begin{proof}\,
First, we consider the eigenfunctions of  (\ref{1.6}), which are even with respect to both
variables. They belong to the space $\wt H^1(\Omega^{++},\Gamma^+)$ and solve the following
problem:
\begin{equation}\label{eq1Gammaeven}
\aligned
\Delta u=&0&& \mbox{in}\quad \Omega^{++};\\
\partial_{\bf n}u=&\lambda u&& \mbox{on}\quad \Gamma^+;
\qquad \partial_{\bf n}u=0\quad \mbox{on}\quad \partial\Omega^{++}\setminus\Gamma^+.
\endaligned
\end{equation}

Moreover, the eigenvalues $\Lambda_k^e$ of the problem (\ref{eq1Gammaeven})
(complemented by zero, enumerated in the increasing order, and repeated
according to their multiplicity)
are critical values of the Rayleigh quotient
$${\mathcal Q}^+[v]\equiv\frac {\|\nabla v\|_{2,\Omega^{++}}^2}{\|v\|_{2,\Gamma^+}^2}$$
over the space
$H^1(\Omega^{++})$. Consider another Rayleigh quotient
$$
\widetilde{\mathcal Q}[v]\equiv
\frac {\|\nabla v\|_{2,\Omega^{++}}^2}{\|v\|_{2,\partial\Omega^{++}}^2}
$$
on the same space.
Since $\widetilde{\mathcal Q}[v]\le{\mathcal Q}^+[v]$, by the variational
principle its critical values $\widetilde\Lambda_k^e$ (which are also
enumerated in the increasing order and repeated according to their multiplicity) satisfy the
relation $\widetilde\Lambda_k^e\le \Lambda_k^e$. However, by scaling
arguments $\widetilde\Lambda_k^e=2\lambda_k$, where $\lambda_k$ is the $k$-th eigenvalue of
the original problem (\ref{1.6}). Therefore,
an eigenfunction of (\ref{1.6}) even with respect to both variables cannot generate  the least
eigenvalue\footnote{We can suggest another proof of this fact, which is interesting by itself.
Let $u$ be a solution of (\ref{eq1Gammaeven}). We claim that at least one of sets
$\varpi_{\pm}=\Omega^{++}\cap\{u\gtrless0\}$ has a connected component which
 touches $\Gamma^+$ but does not touch one of the coordinate axes.
Indeed, consider a connected component of $\varpi_+$ touching
 $\Gamma^+$ (in view of the condition $\int_{\Gamma^+}u=0$, such a component exists).
If this component touches both axes then any connected component of $\varpi_-$
touching $\Gamma^+$ is separated either from $\{x_1=0\}$
or from $\{x_2=0\}$. To be definite, let $\varpi$ be a connected
component of $\varpi_-$ which touches $\Gamma^+$ but does not touch
$\{x_1=0\}$. Then the function
$$v(x_1,x_2)=u(|x_1|,|x_2|)\cdot\chi_{\varpi}(|x_1|,|x_2|)\cdot{\rm sign}(x_1)$$
 belongs to $\wt H^1(\Omega,\Gamma)$ and provides the same value $\Lambda$ of the Rayleigh quotient
(\ref{quotientQ})  as $u$. If $\Lambda$ minimizes ${\mathcal Q}$ on
$\wt H^1(\Omega,\Gamma)$ then $v$ should be a solution of (\ref{1.6})
 which is impossible by the maximum principle. Unfortunately, this argument
is purely 2-dimensional.}.

 Let us consider the eigenfunctions odd with respect to $x_1$. They lead to the following problem:
\begin{equation}\label{eq1Gammaodd}
\aligned
\Delta u=&0 &&\mbox{in}\quad  \Omega^+;\qquad u=0 \quad\mbox{on}\quad  \{x_1=0\};\\
\partial_{\bf n}u=&\lambda u &&\mbox{on}\quad \partial\Omega^+\setminus\{x_1=0\}.
\endaligned
\end{equation}
We claim that the eigenfunction of (\ref{eq1Gammaodd}) corresponding to the least
 eigenvalue must be of constant sign in
$\Omega^+$. Indeed, the function
$$v(x_1,x_2)=|u(|x_1|,x_2)|\cdot{\rm sign}(x_1)$$
belongs to $\wt H^1(\Omega,\Gamma)$
and provides the same value $\lambda$ of the Rayleigh quotient
\begin{equation}
\label{quotientQ}
{\mathcal Q}[v]\equiv\frac {\|\nabla v\|_{2,\Omega}^2}{\|v\|_{2,\Gamma}^2}
\end{equation}
as $u$.
 If $\lambda$ coincides with the infimum of $\mathcal Q$ on $\wt H^1(\Omega,\Gamma)$,
 then $v$ must be a solution of (\ref{1.6}). By the maximum principle,
 $v$ cannot vanish in $\Omega^+$, and the claim follows. Moreover, since the eigenfunctions
 of (\ref{eq1Gammaodd}) are orthogonal in $L^2(\partial\Omega^+\setminus\{x_1=0\})$, no
eigenfunction except the first one can be positive in $\Omega^+$.

Now we consider the equation
\begin{equation}\label{tanh1}
\tan\Big(\frac {\omega h_1}2\Big)\tanh\Big(\frac {\omega h_2}2\Big)=1,
\end{equation}
which obviously has a unique solution $\omega_1$ in the interval $(0,\frac{2\pi}{h_1})$.
Direct calculation shows that the function
$$
v_1(x)=\sin(\omega_1 x_1)\cosh(\omega_1 x_2)
$$
is positive in $\Omega^+$ and solves the problem (\ref{eq1Gammaodd}) with
$\lambda=\omega_1\tanh\left(\frac {\omega_1 h_2}2\right)$
(the equation (\ref{tanh1}) reflects the equality  of
 $\partial_{\bf n}u/u$ on sides of rectangle).
We  substitute
$z_0=\frac {\omega_1}2\sqrt{h_1h_2}$ and conclude that the least eigenvalue of (\ref{eq1Gammaodd})
is equal to $\frac {2z_0(\sqrt{\kappa})}{\sqrt{h_1h_2}}
 \tanh(\frac {z_0(\sqrt{\kappa})}{\sqrt{\kappa}})$,
where $z_0(\sqrt{\kappa})$ is the root of (\ref{tanh}) with
$\alpha=\sqrt{\kappa}:=\sqrt{\frac {h_1}{h_2}}$.

The eigenfunctions of  (\ref{1.6}), which are odd with respect to $x_2$,
can be constructed in a similar way. The corresponding least eigenvalue is
equal to $\frac {2z_0(\frac{1}{\sqrt{\kappa}})}{\sqrt{h_1h_2}}
\tanh(\sqrt{\kappa}z_0(\frac{1}{\sqrt{\kappa}}))$,
where $z_0(\frac{1}{\sqrt{\kappa}})$
is the root of (\ref{tanh}) with $\alpha=\frac{1}{\sqrt{\kappa}}$.

To complete the proof it suffices to show that the function
$f(\alpha)=z_0 \tanh(\frac {z_0}{\alpha})$ decreases on
$(0,+\infty)$. We claim that, in fact, $\alpha f(\alpha)$ is a decreasing function.
Indeed, differentiation of (\ref{tanh}) after some
transformations yields
$$
\frac d{d\alpha}\left(\alpha f(\alpha)\right)=
\frac {2z_0(1-\tanh^4(\frac {z_0}{\alpha}))}{1+\alpha^2-\tanh^2(\frac {z_0}{\alpha})(1-\alpha^2)}
\cdot \bigg[\frac {\tanh(\frac {z_0}{\alpha})}{1+\tanh^2(\frac {z_0}{\alpha})}-z_0\alpha\bigg].
$$
The fraction here is obviously positive. Further, (\ref{tanh})
 implies $z_0\alpha>\frac {\pi}4$. Therefore,
$$
\frac {\tanh(\frac {z_0}{\alpha})}{1+\tanh^2(\frac {z_0}{\alpha})}-z_0\alpha<
\frac 12-\frac {\pi}4<0,
$$
and the claim follows.
\end{proof}

Let us now consider three-dimensional
generalizations of Theorems \ref{rectfull1} and \ref{rectfull2}.

\begin{Theorem}
\label{4.1}
Let
$\Omega=(-\frac{h_1}2,\frac{h_1}2)
\times (-\frac{h_2}2,\frac{h_2}2)\times (-\frac{h_3}2,\frac{h_3}2)$
and let $\Gamma=\partial\Omega$.
Then the exact constant in (\ref{eq1Omega}) is equal to $\frac 1{\pi}\max\{h_1;h_2;h_3\}$.
\end{Theorem}
\begin{proof}\,
The proof is similar to the proof of Theorem \ref{rectfull1}.
Instead of (\ref{cot}) we obtain the equation
\begin{equation*}
\frac {h_1h_2}2\cot\Big(\frac {\omega h_3}2\Big)+
\frac {h_1h_3}2\cot\Big(\frac {\omega h_2}2\Big)+
\frac {h_2h_3}2\cot\Big(\frac {\omega h_1}2\Big)+
\frac 2{\omega}\,(h_1+h_2+h_3)=0.
\end{equation*}
The unique solution of this equation in
$(0,\frac{2\pi}h)$ (where $h=\max\{h_1;h_2;h_3\}$) is greater than
$\frac{\pi}h$, and the statement follows.
\end{proof}

\begin{Theorem}
Let $\Omega$ and $\Gamma=\partial\Omega$ be as in Theorem \ref{4.1}. To be definite, assume
that $\max\{h_1;h_2;h_3\}=h_3$. Then the exact constant in (\ref{eq1Gamma}) is equal to
$\big(\frac {2z_1}{h_1} \tanh(z_1)\big)^{-\frac 12}$, where $(z_1,z_2)$ is a unique solution
of the system
\begin{multline}
\label{4.2}
\frac {z_1}{h_1}\tanh(z_1)=\frac {z_2}{h_2}\tanh(z_2)=\\
=\frac {z_1}{h_1}\sqrt{1+\frac {\tanh^2(z_1)}{\tanh^2(z_2)}}\cdot
\cot\bigg(\frac{z_1h_3}{h_1}\sqrt{1+\frac {\tanh^2(z_1)}{\tanh^2(z_2)}}\bigg),
\end{multline}
such that  $0<\frac{z_1h_3}{h_1}\sqrt{1+\frac {\tanh^2(z_1)}{\tanh^2(z_2)}}<\frac {\pi}2$.
\end{Theorem}

\begin{proof}\,
By the same arguments as in the proof of Theorem \ref{rectfull2},
 we conclude that the eigenfunction of (\ref{1.6}) corresponding to the least
eigenvalue must be odd with respect to one of the coordinate axes
(e.g., with respect to $x_3$). This gives the following analog
of (\ref{eq1Gammaodd}):
\begin{equation}\label{4.3}
\aligned
\Delta u=&0 &&\mbox{in}\quad  \widehat\Omega^+:=(-\frac{h_1}2,\frac{h_1}2)
\times (-\frac{h_2}2,\frac{h_2}2)\times (0,\frac{h_3}2);\\
 u=&0&& \quad\mbox{on}\quad  \{x_3=0\};\\
\partial_{\bf n}u=&\Lambda u &&\mbox{on}\quad \partial\widehat\Omega^+\setminus\{x_3=0\}.
\endaligned
\end{equation}

Repeating  the proof of Theorem \ref{rectfull2}, we find that the least
eigenvalue is generated by the eigenfunction
 of (\ref{4.3}), which is positive in
$\widehat\Omega^+$.

It is easy to see that the equation
\begin{subequations}
\label{4.5}
\begin{equation}
\label{4.5a}
\mu\tanh\Big(\frac {\mu h_1}2\Big)=\nu\tanh\Big(\frac {\nu h_2}2\Big)
\end{equation}
defines an increasing function $\nu=\nu(\mu)$ on ${\mathbb R}_+$, and the equation
\begin{equation}
\label{4.5b}
\mu\tanh\Big(\frac {\mu h_1}2\Big)=\sqrt{\mu^2+\nu^2(\mu)}\cdot\cot\left(\sqrt{\mu^2+\nu^2(\mu)}\
\frac {h_3}2\right)
\end{equation}
\end{subequations}
has a unique solution $\mu_0$ such that $\sqrt{\mu_0^2+\nu^2(\mu_0)}\ \frac {h_3}2<\frac {\pi}2$.
Direct calculation shows that the function
$$
U_1(x)=\cosh(\mu_0 x_1)\cosh(\nu(\mu_0) x_2)
\sin\left(\sqrt{\mu_0^2+\nu^2(\mu_0)}\, x_3\right)
$$
is positive in $\widehat\Omega^+$ and solves the problem (\ref{4.3}) with
$\Lambda=\mu_0\tanh\big(\frac {\mu_0 h_1}2\big)$
(note that (\ref{4.5}) reflects the boundary conditions on the sides of parallelepiped).

Substituting $z_1=\frac {\mu_0h_1}2$, $z_2=\frac {\nu(\mu_0)h_2}2$, we conclude
that the least eigenvalue of (\ref{4.3}) equals $\frac {2z_1}{h_1} \tanh(z_1)$, where
$(z_1,z_2)$ is the solution of (\ref{4.2}).

The eigenfunctions odd with respect to other coordinates can be constructed
quite analogously. However, some additional calculations
show that  $U_1$ is the best eigenfunction provided that $h_3$ is the longest
edge of $\Omega$. Thus, the statement follows.\end{proof}


\section{Exact constants for isosceles right triangles}

In this section, $\Omega$ is an isosceles right triangle.
We find sharp constants in (\ref{eq1Omega}) and (\ref{eq1Gamma})
for the following three cases: $\Gamma$ is a leg, $\Gamma$ is the union of two legs,
and $\Gamma$ is the hypotenuse.
Finding exact constants in (\ref{eq1Omega}) and (\ref{eq1Gamma}) for the case $\Gamma=\partial\Omega$ remains an open problem.

\begin{figure}[h!]
\label{isosceles}
\centerline{\includegraphics[width=3.8in,height=1.6in]{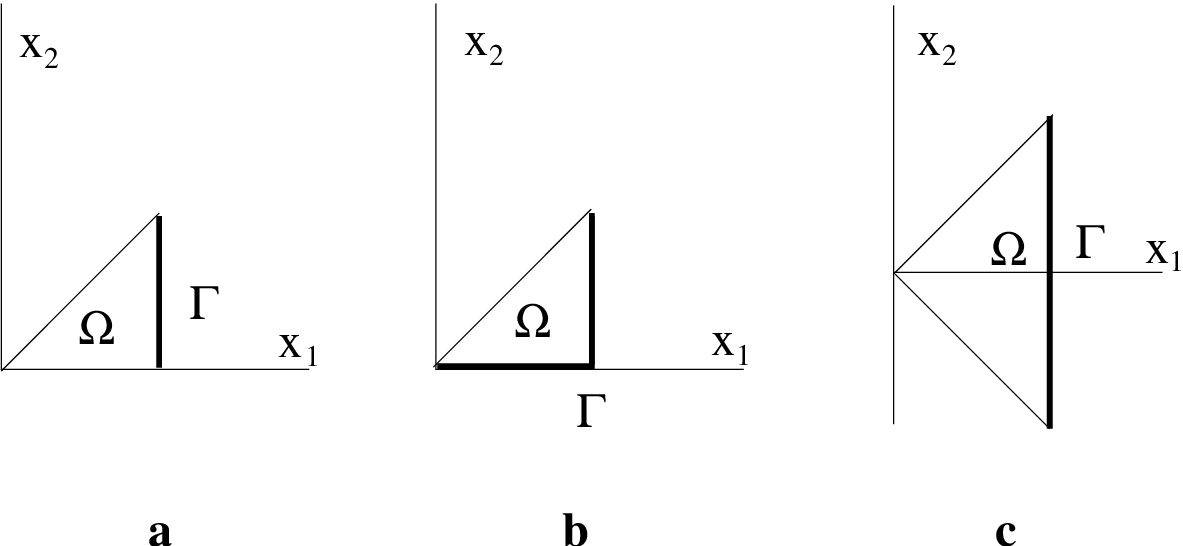}}
\caption{Right triangles}
\end{figure}

\subsection{Case 1: $\Gamma$ is a leg}

It is convenient to select the coordinate system such that
$\Omega=\{0<x_2<x_1<h\}$ and $\Gamma=\{x_1=h,x_2\in[0,h]\}$ (see Fig.~2a).

\begin{Theorem}\label{threeleg1}
The exact constant in (\ref{eq1Omega}) is equal to
$\widetilde z_0^{-1}h$, where $\widetilde z_0\approx 2.02876$ is a unique root of the equation
\begin{equation}\label{cot1}
z\cot(z)+1=0
\end{equation}
in the interval $(0,\pi)$.
\end{Theorem}

\begin{proof}\,
We use the same arguments as in the proof of Theorem \ref{rectfull1} and
note that the eigenvalues $\lambda_k^{\vartriangle}$
of the problem (\ref{1.5}) (enumerated in the
increasing order and repeated according to their multiplicity)
are critical values of the Rayleigh quotient
$Q[v]\equiv\frac {\|\nabla v\|_{2,\Omega}^2}{\|v\|_{2,\Omega}^2}$
over the space $\wt H^1(\Omega,\Gamma)$.

Let us consider $Q[v]$ on the whole space $H^1(\Omega)$.
In accordance with the variational
principle, the corresponding critical values
$\widetilde\lambda_k^{\vartriangle}$ (which are also enumerated in the increasing order
and repeated according to their multiplicity) satisfy the relation
$\widetilde\lambda_k^{\vartriangle}\le \lambda_k^{\vartriangle}\le
\widetilde\lambda_{k+1}^{\vartriangle}$.
Therefore, if  the interval
$(\widetilde\lambda_0^{\vartriangle},\widetilde\lambda_1^{\vartriangle})$
contains an eigenvalue of the problem (\ref{1.5}),
then it must be $\lambda_0^{\vartriangle}$.

Note that $\widetilde\lambda_k^{\vartriangle}$ are eigenvalues
of the conventional Neumann problem
\begin{equation}\label{Neumann}
-\Delta u=\lambda u\quad  \mbox{in}\quad  \Omega;\qquad \partial_{\bf n}u=0\quad
\mbox{on}\quad  \partial\Omega.
\end{equation}
By even reflection with respect to the line $\{x_1=x_2\}$,
we conclude that any eigenfunction of (\ref{Neumann}) is an eigenfunction
of the Neumann problem in the square $(0,h)\times(0,h)$. In particular,
$\widetilde\lambda_0^{\vartriangle}=0$ corresponds to
the eigenfunction $\widetilde u_0\equiv1$, and
$\widetilde\lambda_1^{\vartriangle}=\left(\frac{\pi}h\right)^2$ corresponds to the
eigenfunction $\widetilde u_1(x)=\cos(\frac {\pi x_1}h)+\cos(\frac {\pi x_2}h)$.

Direct calculation shows that the function
$$
\widetilde v_0(x)=\cos\Big(\frac {\widetilde z_0 x_1}h\Big)
+\cos\Big(\frac {\widetilde z_0 x_2}h\Big)
$$
solves the problem (\ref{1.5})
 with $\lambda=(\frac {\widetilde z_0}h)^2$
  (the equation (\ref{cot1}) is just the condition
$\int_{\Gamma}\widetilde v_0\,ds=0$). Thus, we conclude that
$\lambda_0^{\vartriangle}=(\frac {\widetilde z_0}h)^2$, and the statement
follows.\end{proof}

\begin{Theorem}\label{threeleg2}
 The sharp
constant in (\ref{eq1Gamma}) is equal to
$\left(\frac {\widehat z_0}h\tanh(\widehat z_0)\right)^{-\frac 12}$
where $\widehat z_0\approx 2.3650$ is a unique root of the equation
\begin{equation}
\label{tanh2}
\tan(z)+\tanh(z)=0
\end{equation}
in the interval $(0,\pi)$.
\end{Theorem}

\begin{proof}\,
Let $v$ be the minimizer of the Rayleigh quotient
${\mathcal Q}[v]\equiv\frac {\|\nabla
v\|_{2,\Omega}^2}{\|v\|_{2,\Gamma}^2}$ over the space
$\wt H^1(\Omega,\Gamma)$.
Since the trace $v|_{\Gamma}$ is orthogonal to $1$ in $L^2(\Gamma)$, 
its primitive $F(y)=\int_0^{y} v(h,t)\,dt$ vanishes at $0$ and at $h$.

We claim that $F$ has constant sign on $(0,h)$. Indeed,
otherwise there exists $\alpha\in(0,1)$ such that
\begin{equation}
\int\limits_0^{\alpha h} v(h,x_2)\,dx_2=\int\limits_{\alpha h}^h
v(h,x_2)\,dx_2=0.
\label{split}
\end{equation}
We split $\Omega$ into two subdomains
$$
\Omega_1=\Omega\cap\{x_2<\alpha x_1\}\quad{\rm and}\quad
\Omega_2=\Omega\cap\{x_2>\alpha x_1\}.
$$
Let us
compare the respective Rayleigh quotients
$$
{\mathcal Q}_j[v]=\frac {\|\nabla v\|_{2,\Omega_j}^2}{\|v\|_{2,\Gamma_j}^2},
\qquad j=1,2,
$$
where  $\Gamma_j:=\Gamma\cap\partial\Omega_j$.
Assume that ${\mathcal Q}_1[v]\le {\mathcal Q}_2[v]$.
We define the function
 $V(x_1,x_2)=v(x_1,\alpha x_2)$. Then $\int_{\Gamma}V\,dx_2=0$ by
(\ref{split}),
and
$$
{\mathcal Q}[V]=
\frac {\int\limits_{\Omega}\big(V_{x_1}^2+V_{x_2}^2\big)\,dx_1dx_2}
{\int\limits_{\Gamma}V^2\,dx_2}
=\frac
{\int\limits_{\Omega_1}\big(v_{x_1}^2+\alpha^2v_{x_2}^2\big)\,dx_1dx_2}{\int\limits_{\Gamma_1}v^2\,dx_2}\le
{\mathcal Q}_1[v]\le
{\mathcal Q}[v].
$$
Moreover, the first inequality in the above relations is strict. Indeed, the converse
implies $v_{x_2}\equiv0$, what is impossible since $v$ is harmonic.


Analogously, if ${\mathcal Q}_2[v]\le {\mathcal Q}_1[v]$, then we consider
the function
 $V(x_1,x_2)=v(x_1,(1-\alpha) x_2+\alpha h)$ and again find that
${\mathcal Q}[V]<{\mathcal Q}[v]$. Thus, the claim follows.\medskip

Further, direct calculation shows that the function
$$
\widetilde v_1(x)=\cos\Big(\frac {\widehat z_0 x_1}h\Big)\cosh\Big(\frac {\widehat z_0 x_2}h\Big)+
\cosh\Big(\frac {\widehat z_0 x_1}h\Big)\cos\Big(\frac {\widehat z_0 x_2}h\Big)
$$
solves the problem (\ref{1.6}) with $\lambda=\frac {\widehat z_0}h\tanh(\widehat z_0)$ 
(the equation (\ref{tanh2}) is just the condition $\int_{\Gamma}\widetilde v_1\,ds=0$). 
Since $\frac {\pi}2<\widehat z_0<\pi$, $\widetilde v_1|_{\Gamma}$ is monotone decreasing 
in $x_2$ and, in particular, 
the primitive of $\widetilde v_1|_{\Gamma}$ has constant sign.

Finally, any other eigenfunction $v$ of the problem (\ref{1.6}) can be chosen orthogonal to $1$ and to 
$\widetilde v_1$ in $L^2(\Gamma)$. 
But if $F$ (the primitive of $v|_{\Gamma}$) has a
constant sign (e.g., positive), then we arrive at a contradiction because
$$
\int\limits_0^{h} v(h,t)\widetilde v_1(h,t)\,dt=
-\int\limits_0^{h} F(h,t)\widetilde v_1'(h,t)\,dt>0.
$$
Thus, $\widetilde v_1$ should minimize ${\mathcal Q}[v]$ over
$\wt H^1(\Omega,\Gamma)$. The proof is complete.
\end{proof}

\subsection{Case 2:  $\Gamma$ is the union of two legs}
Let (see Fig.~2b)
$$\Omega=\{0<x_2<x_1<h\}\ \ {\rm and}\ \ 
\Gamma=\{x_1=h,x_2\in[0,h]\}\cup\{x_2=0,x_1\in[0,h]\}.
$$
\begin{Theorem}
\label{three2legs1}
The sharp constant in (\ref{eq1Omega}) is equal to $\frac h{\pi}$.
\end{Theorem}

\begin{proof}\,
Making even reflection with respect to the line $\{x_1=x_2\}$,
 we conclude that any eigenfunction of (\ref{1.5})
is an eigenfunction of the same problem in the square domain $\Omega'=(0,h)\times(0,h)$
with $\Gamma=\partial\Omega'$.
The corresponding eigenvalue problem has been studied
 in Theorem \ref{rectfull1}, which states that the least positive eigenvalue
 is equal to $(\frac {\pi}h)^2$.
The  corresponding eigenspace has the dimension $2$ and contains the function
$\cos(\frac {\pi}h x_1)+\cos(\frac {\pi}h x_2)$ which solves the original
 problem in the triangle.
 \end{proof}

\begin{Theorem}
\label{three2legs2}
The sharp constant in (\ref{eq1Gamma}) is equal to
$\left(\frac {2z_0}h \tanh(z_0)\right)^{-\frac 12}$, where $z_0\approx 0.93755$
is a unique root of the equation (\ref{tanh}) in  $(0,\frac {\pi}2)$
for $\alpha=1$.
\end{Theorem}

\begin{proof}\,
Again, we use even reflection with respect to the line $\{x_1=x_2\}$ and reduce our problem
to the problem in the square domain $(0,h)\times(0,h)$. By Theorem
\ref{rectfull2}, we conclude that the least positive eigenvalue is equal to
 $\frac {2z_0}h \tanh(z_0)$.
The  corresponding eigenspace has the dimension $2$ and contains the function
\begin{equation*}
\sin\xi_1\cosh\xi_2+
\cosh\xi_1
\sin\xi_2,
\end{equation*}
 where
$\xi_1=z_0\Big(\frac {2x_1}h-1\Big)$ and $\xi_2=z_0\Big(\frac {2x_2}h-1\Big)$.
It solves the original problem in the triangle.
\end{proof}

\subsection{Case 3:  $\Gamma$ is the hypotenuse}

In this case, it is convenient to select the coordinate system such that (see Fig.~2c)
$$\Omega=\{0<|x_2|<x_1<h\}\quad{\rm and}\quad \Gamma=\{x_1=h, x_2\in[-h,h]\}.
$$
\begin{Theorem}
\label{threehyp}
The sharp constant in (\ref{eq1Omega}) is equal to
$\widetilde z_0^{-1}h$, where $\widetilde z_0$ is defined in Theorem \ref{threeleg1}.
\end{Theorem}
\begin{proof}\,
First, we consider the eigenfunctions of  (\ref{1.5}), which are
odd with respect to $x_2$. Then, $\mu=0$ and we arrive at the
following problem:
\begin{equation}
\label{eq1.1odd-a}
\aligned
-\Delta u=&\lambda u &&\mbox{in}\quad
\widetilde\Omega^+:=\{0<x_2<x_1<h\},\\
 u=&0 \quad&&\mbox{on}\quad\{x_2=0\},\\
\partial_{\bf n}u=&0&&\mbox{on}\quad \partial\widetilde\Omega^+\setminus\{x_2=0\}.
\endaligned
\end{equation}
As in Theorems \ref{three2legs1} and \ref{three2legs2}, we use even reflection
with respect to the line $\{x_1=x_2\}$ and reduce (\ref{eq1.1odd-a})
to the problem in $(0,h)\times(0,h)$. Thus, we conclude
that the least eigenvalue of the problem (\ref{eq1.1odd-a}) is equal to
$\frac 12(\frac {\pi}h)^2$ and corresponds to the eigenfunction
$$
\widehat u_0(x)=\sin\left(\frac {\pi x_1}{2h}\right)\sin\left(\frac {\pi x_2}{2h}\right).
$$
Next, we consider the eigenfunctions, which are even
with respect to $x_2$. Then, we arrive at the problem (\ref{1.5})
in $\widetilde\Omega^+$
which has been already solved in Theorem \ref{threeleg1}.

To complete the proof we compare
$\frac 12(\frac {\pi}h)^2$ and $(\frac {\widetilde z_0}h)^2$.
It is easy to check that
$\frac {\pi}{\sqrt{2}}\cdot\cot(\frac {\pi}{\sqrt{2}})<-1$.
 Since $t\mapsto t\cdot\cot(t)$ is a decreasing function on $(0,\pi)$,
this means that $\frac {\pi}{\sqrt{2}}>\widetilde z_0$, and the statement follows.
\end{proof}

\begin{Theorem}
 The exact constant in (\ref{eq1Gamma}) is equal to $h^{\frac 12}$.
\end{Theorem}
\begin{proof}\,
First, we consider eigenfunctions of the problem (\ref{1.6})
even with respect to $x_2$. Then we arrive at the problem (\ref{1.6})
in $\widetilde\Omega^+$ which is solved in Theorem \ref{threeleg2}.

Further, let us consider the eigenfunctions, which are odd with respect
to $x_1$. We arrive at the following problem in $\widetilde\Omega^+$:
\begin{equation} \label{eq1Gammaodd-a}
\aligned
\Delta u &=0\quad \mbox{in}\quad\widetilde\Omega^+;&
\partial_{\bf n}u &=0 \quad \mbox{on}\quad \{x_1=x_2\};\\
u &=0 \quad  \mbox{on}\quad  \{x_2=0\};&
 \partial_{\bf n}u &=\lambda u\quad\mbox{on}\quad  \{x_1=h\}.
\endaligned
\end{equation}
Direct calculation shows that the function $x_1x_2$
is positive in $\widetilde\Omega^+$ and solves the problem (\ref{eq1Gammaodd-a}) with
$\lambda=\frac 1h$. By the same arguments as we used for  the problem (\ref{eq1Gammaodd}),
 it can be shown that this function  corresponds to the least eigenvalue.

To complete the proof we compare
$\frac {\widehat z}h\tanh(\widehat z_0)$ and $\frac 1h$.
Since $\widehat z_0>\frac {\pi}2$,
we have $\widehat z_0\tanh(\widehat z_0)>\frac {\pi}2\tanh(\frac {\pi}2)>1$,
and the statement follows.
\end{proof}

\section{An application of the estimates}
Exact values of $ C_1(\Omega,\Gamma)$ and $ C_2(\Omega,\Gamma)$
for rectangular domains and  right  triangles presented in  Theorems
\ref{rectside}--\ref{rectfull2} and \ref{threeleg1}--\ref{threehyp}   yield
guaranteed bounds of constants in analogous inegualities for arbitrary nondegenerate
triangles and convex quadrilaterals. The corresponding constants  are
estimated by standard techniques based on affine equivalent transformations (see \cite{MaRe}).
Estimates of the constant $C_2$ can be also obtained by monotonicity
arguments. Indeed, if $\Omega_1$ and $\Omega_2$ are two Lipscitz domains
such that $\Omega_1\subset\Omega_2$ and $\Gamma$ belongs to
 $\partial\Omega_1\cap\partial\Omega_2$, then $C_2(\Omega_1,\Gamma)\geq C_2(\Omega_2,\Gamma)$.

Estimates (\ref{eq1Omega}) and (\ref{eq1Gamma}) with known constants can be used
in various problems arising in  quantitative analysis of  partial differential equations.
For example, let $\Omega$ be a polygonal domain covered by a simplicial mesh ${\mathcal T}_h$,
which cells $T_i$ ($i=1,2,...,N$) have a character size $h$. Assume that a function $w_h\in H^1(\Omega)$ satisfies the condition
$\mean{w_h}_{E_{ij}}=0$ on any edge $E_{ij}=~\partial{T_i}~\cap~\partial{T_j}$. Then, we find that
\be
\|w_h\|_\Omega\,\leq\, C\|\nabla w_h\|_\Omega,\;{\rm where}\; C=\max\limits_{i=1,2,...,N}\,\min\limits_{E_{ij}\subset \partial T_i} \,C_1(T_i,E_{ij}).
\ee
Estimates of this type contain explicitly known constants (which are proportional to $h$)
and allow us to derive
fully computable a posteriori estimates of the accuracy of approximate solutions (in \cite{MaRe}
this question is studied in the context of parabolic type equations).

However, a posteriori analysis of approximation errors is only
 one possible area of application.
Below, we briefly discuss another example, in which the estimates are used to
deduce a computable a priori estimate of the difference between exact solutions
of two boundary value problems.
Estimates of this type show when it is worth using a simplified mathematical model instead
of the original (complicated) one.

Consider the following {\em Problem $\cP$}: find
$u\in H^1(\Omega)$ such that
$$
u-u_0\in V_0:=\{w\in H^1(\Omega): w|_{\Gamma^D}=0\}
$$
and
\begin{equation}
\label{ap5}
\IntO A\nabla u\cdot\nabla w\,dx=\IntO fw\,dx+\int\limits_{\Gamma_N}Fw\,ds,
\qquad\forall w\in V_0,
\end{equation}
where  $\partial\Omega$ consists of two measurable nonintersecting parts
$\Gamma^D$ and $\Gamma^N$ associated with Dirichlet and Neumann boundary conditions, respectively,
$f\in L^2(\Omega)$,  $F~\in~ L^2(~\Gamma^N~)$, and the trace of function  $u_0\in H^1(\Omega)$
defines the Dirichlet
boundary condition.  We assume
that the matrix $A$ is symmetric, bounded, and satisfies the uniform ellipticity condition
$ A\xi\cdot\xi\,\geq\, c|\xi|^2$, where $c$ is a positive constant and the dot stands
for the scalar product of vectors.

Let  $\Omega$ be divided into a finite set $\cO$  of "simple" nonoverlapping
subdomains
$\Omega_i$ (e.g., they can be cells of
a certain polygonal  mesh). Each $\Omega_i$ belongs to one of the following three subsets:
\begin{eqnarray*}
&&\cOD:=\{\Omega_i\subset \Omega\,\mid\, \partial\Omega_i\cap\Gamma^D=:\Gamma^D_i\not=\emptyset\},\\
&&\cON:=\{\Omega_i\subset \Omega\,\mid\, \partial\Omega_i\cap\Gamma^N=:\Gamma^N_i\not=\emptyset\},\quad
 \cOI:=\cO\setminus(\cOD\cup\cON).
\end{eqnarray*}
Here, $\cOI$ contains interior subdomains, $\cOD$ contains subdomains associated with
$\Gamma^D$, and elements of $\cON$ are the  subdomains associated with  $\Gamma^N$.
Then, $$\overline\Omega=\overline\Omega^D\cup\overline\Omega^I\cup\overline\Omega^N,
$$
where $\Omega^D$, $\Omega^N$, and $\Omega^I$ consist of $\Omega_i$ from $\cOD$, $\cON$, and $\cOI$,
respectively.
In the general case,  $\Omega_i$ may have common boundaries with
 $\Gamma_N$ and $\Gamma_D$, but then
we can subdivide it into two and obtain a decomposition of the above type, which subdomains
form three sets: $\cOI$, $\cON$, and $\cOD$ (see Fig.~3).

\begin{figure}[h!]
\label{domainNR}
\centerline{\includegraphics[width=2.7in]{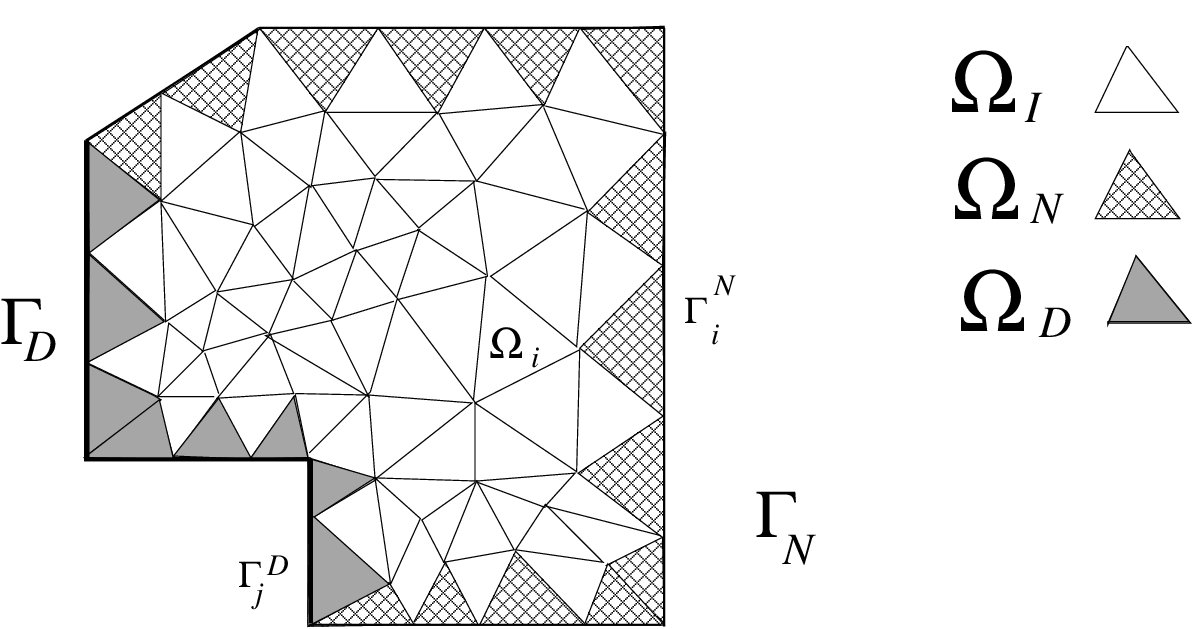}}
\caption{}
\end{figure}

Now, instead of $\cP$ we consider a modified {\em problem $\wh \cP$}:
 find
$\wh u\in H^1(\Omega)$ such that
$$
\wh u-\wh u_0\in V_0:=\{w\in H^1(\Omega): w|_{\Gamma^D}=0\}
$$
and
\begin{equation}
\label{ap6}
\IntO A\nabla \wh u\cdot\nabla w\,dx=\IntO \wh fw\,dx+\int\limits_{\Gamma_N}\wh Fw\,ds,
\qquad\forall w\in V_0,
\end{equation}
where the functions $\wh f$, $\wh F$, and $\wh u_0$ are simplified counterparts of $f$, $F$, and $u_0$, respectively.
Simplifications are done on subdomains and can be motivated by different reasons, e.g., they may be applied
in order to avoid difficult integration procedures, which are necessary
if the functions have a complicated behavior with many local details.

Our goal is to
 deduce an estimate of the difference between the exact solutions of these two problems
in terms of the norm
$$
\Normt{u-\wh u}^2:=\IntO A \nabla(u-\wh u)\cdot \nabla(u-\wh u)\,dx,
$$
which presents the {\em error of simplification}. If this error is insignificant,
then the problem (\ref{ap6}) can be successfully used instead of (\ref{ap5}).

In the derivation, we use the Poincare estimate (\ref{1.1}) for cells in the set $\Omega_I$
with the
respective constant $C_P(\Omega_i)$ (or its upper bound $\frac{{\rm diam }\Omega_i}{\pi}$). For the cells in $\Omega_N$,
we use the estimate (\ref{eq1Gamma}). If the cells are formed by triangles, then  $\Gamma^N_i$
may contain either one or two sides of triangles  (see Fig. ~3). Therefore, we can use estimates, which
follow from Theorems \ref{threeleg1} and \ref{three2legs1}. Analogously, for the cells in $\Omega_D$,
we use the estimate (\ref{eq1Omega}) and
Theorems \ref{threeleg2} and \ref{three2legs2} related to the cases,
in which functions have zero mean value either on one or on two sides of triangles.

We define the quantities
\begin{eqnarray}
\label{ap11}
&&D^2_1:=\!\sum\limits_{\Omega_i\in \cO}\!
 {\mathbb C}^2_i\|f-\wh f\|^2_{2,\Omega_i}\ \ {\rm and}\ \ 
D^2_2:=\!\sum\limits_{\Omega_i\in \cON}\!
 C_2(\Omega_i,\Gamma^N_i)^2\|F-\wh F\|^2_{2,\Gamma^N_i},
\end{eqnarray}
where
\begin{eqnarray}
\label{ap14}
 {\mathbb C}_i=\left\{
 \begin{array}{ll}
 C_P(\Omega_i)\;&{\rm if }\quad\Omega_i\in \cOI\cup\cON,\\
C_1(\Omega_i,\Gamma^D_i)\ &{\rm if }\quad\Omega_i\in \cOD.
\end{array}
\right.
\end{eqnarray}
Finding the quantities is a very easy task. It is reduced to integration of known functions
and { does not require solving a boundary value problem}. The theorem below shows that
a guaranteed majorant of $\Normt{u-\wh u}$ can be expressed throughout $D_1$,  $D_2$
and easily computable quantities
\begin{eqnarray}
\label{ap19}
&{\mathcal I}_0= \sum\limits_{\Omega_i\in \cOD}\mean{u_0-\wh u_0}_{\Gamma^D_i}
\displaystyle\int\limits_{\Omega_i}(f-\wh f)\,dx,\\
\label{ap20}
&{\mathcal I}_1(\phi)=\displaystyle\int\limits_{\Omega}(f-\wh f)\phi\,dx,\qquad{\mathcal I}_2(\phi)=
\displaystyle\int\limits_{\Gamma^N }(F-\wh F)\,\phi\,ds,
\end{eqnarray}
where $\phi$ is an arbitrary function in 
$H^1(\Omega)$ such that $\phi=u_0-\wh u_0$ on $\Gamma^D$.
\begin{Theorem}
\label{Thuwhu}
Let $u$ and $\wh u$ be the solutions of (\ref{ap5}) and (\ref{ap6}),
respectively, and
\begin{eqnarray}
\label{ap15}
&\mean{f-\wh f}_{\Omega_i}=0\qquad&\,\forall\, \Omega_i\in \cOI\cup\cON,\\
\label{ap16}
&\mean{F-\wh F}_{\Gamma^N_i}=0\qquad &\, \forall \,\Omega_i\in \cON.
\end{eqnarray}
Then
\begin{equation}
\label{ap17}
\Normt{u-\wh u}\leq\,\rho_1+
\sqrt{\rho_2+\rho^2_1},
\end{equation}
where
\begin{equation}
\label{ap18}
2\rho_1=\frac{D_1+D_2}{\sqrt{c}}+\Normt{\phi}\quad{\rm and}\quad
\rho_2={\mathcal I}_0+{\mathcal I}_1(\phi)+{\mathcal I}_2(\phi).
\end{equation}
\end{Theorem}

\begin{proof} From (\ref{ap5}) and  (\ref{ap6}) it follows that
\begin{multline}
   \int\limits_\Omega A ( \nabla u - \nabla \wh u ) \cdot \nabla w \, dx\\
   =\int\limits_{\Omega}(f-\wh f)w\,dx+
    \int\limits_{\Gamma^N }(F-\wh F)\,w\,ds\quad\forall w\in V_0.
   \label{ap21}
\end{multline}
Since $w=u-\wh u-\phi\in V_0$ we can use it as a test function. Then
\begin{multline}
 \label{ap22}
 \int\limits_\Omega A  \nabla( u -  \wh u ) \cdot \nabla w \, dx=\Normt{u-\wh u}^2+
  \int\limits_\Omega A  \nabla (u -  \wh u ) \cdot \nabla\phi\, dx\\
\ge \Normt{u-\wh u}^2-\Normt{u-\wh u}\,\Normt{\phi}.
\end{multline}
Consider the first term in the right-hand side of (\ref{ap21}):
\begin{multline}
 \label{ap23}
 \int\limits_\Omega (f-\wh f)w\, dx\\
 =\sum\limits_{\Omega_i\in \cOI\cup\cON}\int\limits_{\Omega_i}(f-\wh f)(u-\wh u)dx+
 \sum\limits_{\Omega_i\in \cOD}\int\limits_{\Omega_i}(f-\wh f)(u-\wh u)dx+{\mathcal I}_1(\phi).
 \end{multline}
 The terms of the first sum in (\ref{ap23}) are estimated using (\ref{ap15}) and (\ref{1.1}):
 \begin{multline*}
\int\limits_{\Omega_i}(f-\wh f)(u-\wh u)dx= 
\int\limits_{\Omega_i}(f-\wh f)(u-\wh u-\mean{u-\wh u}_{\Omega_i})\,dx\\
 \leq 
 C_P(\Omega_i)\|f-\wh f\|_{2,\Omega_i}\| \nabla (u -  \wh u) \|_{2,\Omega_i},
\qquad \Omega_i\in \cOI\cup\cON.
 \end{multline*}
 The terms of the second sum are estimated with the help of (\ref{eq1Omega}) as follows:
 \begin{multline*}
\int\limits_{\Omega_i}(f-\wh f)(u-\wh u)\,dx= 
\int\limits_{\Omega_i}(f-\wh f)(u-\wh u-\mean{u-\wh u}_{\Gamma^D_i})\,dx\\
 +
\int\limits_{\Omega_i}(f-\wh f)\mean{u-\wh u}_{\Gamma^D_i}\,dx\le 
C_1(\Omega_i,\Gamma^D_i)\|f-\wh f\|_{2,\Omega_i}
 \| \nabla (u -  \wh u) \|_{2,\Omega_i}\\
 +
\mean{u-\wh u}_{\Gamma^D_i}\int\limits_{\Omega_i}(f-\wh f)dx,\qquad \Omega_i\in \cOD.
 \end{multline*}
 Summing up these estimates and using  (\ref{ap19}) we obtain
\begin{multline}
 \label{ap26}
 \int\limits_\Omega (f-\wh f)w\ dx\leq
\sum\limits_{\Omega_i\in \cO}{\mathbb C}_i\|f-\wh f\|_{2,\Omega_i}
   \| \nabla (u -  \wh u) \|_{2,\Omega_i}+ {\mathcal I}_0+{\mathcal I}_1(\phi)\\
\le \frac {D_1}{\sqrt{c}}\Normt{u-\wh u}+ {\mathcal I}_0+{\mathcal I}_1(\phi).
\end{multline}

By means of (\ref{ap16}) and (\ref{eq1Gamma}) we find that
\begin{equation*}
\int\limits_{\Gamma^N_i }(F-\wh F)(u-\wh u)\,ds\leq
 C_2(\Omega,\Gamma^N_i)\|F-\wh F\|_{2,\Gamma^N_i}
\|\nabla (u - \wh u)\|_{2,\Omega_i},
\end{equation*}
and (cf. (\ref{ap20}))
\begin{equation}
\label{ap28}
 \int\limits_{\Gamma^N }(F-\wh F)w\,ds
 \leq \frac {D_2}{\sqrt{c}}\Normt{u-\wh u}+{\mathcal I}_2(\phi).
\end{equation}
 Now (\ref{ap21}), (\ref{ap23}), (\ref{ap26}), and (\ref{ap28}) yield the estimate
\begin{equation}
 \label{ap29}
 \Normt{u-\wh u}^2\leq
 2\rho_1 \Normt{u-\wh u}+\rho_2,
\end{equation}
 where the quantities $\rho_1$ and $\rho_2$ are defined by (\ref{ap18}).

The quadratic inequality (\ref{ap29}) easily implies (\ref{ap17}).
\end{proof}
\begin{Remark}
It is worth noting that the quantity $\rho_2$ may be negative. However, the right hand
side of (\ref{ap17}) is always nonnegative.
Moreover, it vanishes if and only if $u=\wh u$ (i.e., if two problems are identical).
\end{Remark}

Theorem \ref{Thuwhu} presents the most general form of the estimate.
If $\wh u_0=u_0$, then it can be  simplified. Indeed, in this case
one can choose $\phi\equiv0$, and (\ref{ap17}) is reduced to
\begin{equation*}
\Normt{u-\wh u}\leq\frac{D_1+D_2}{\sqrt{c}}.
\end{equation*}
Moreover, in this case we can replace $C_1(\Omega_i,\Gamma^D_i)$ in (\ref{ap14}) by a smaller
constant $C_F(\Omega_i,\Gamma^D_i)$ such that
\begin{equation}
\label{GammaD}
\|w\|_{2,\Omega}\leq C_F(\Omega_i,\Gamma^D_i)\|\nabla w\|_{2,\Omega},\qquad
\forall w\in H^1(\Omega_i)\, : \, w|_{\Gamma^D_i}=0.
\end{equation}
For simple domains such as rectangles or isoceles right
 triangles these constants are well-known.

Finally, we note that  similar analysis  can be performed for other
differential equations associated with the pair of conjugate operators ${\rm grad}$
and $-\dvg$.\medskip

\section{Acknowledgments}
The first author was supported by RFBR grants 11-01-00825 and 14-01-00534 and by St.Peterburg
University grant 6.38.670.2013. The second author was supported by RFBR grant 11-01-00531-a.
We thank N. Filonov for useful comments. Also, we are grateful to the anonymous reviewer
and to S. Kr\"omer who found gaps in the previous proofs of Theorem 3.2.

\end{document}